\documentclass[12pt]{elsarticle}

\usepackage{amsxtra}
\usepackage{amsfonts}
\usepackage[latin1]{inputenc}
\usepackage{graphicx}
\usepackage{amsmath,amssymb,latexsym,amsthm,mathrsfs}
\usepackage[all]{xy}

\newtheorem{theorem}{Theorem}[section]
\newtheorem{thm}[theorem]{Theorem}
\newtheorem{lemma}[theorem]{Lemma}
\newtheorem{corollary}[theorem]{Corollary}

\newtheorem{defn}[theorem]{Definition}
\newtheorem{proposition}[theorem]{Proposition}

\newtheorem{observation}[theorem]{Observation}
\newtheorem{assumption}[theorem]{Assumption}

  \newcommand{\bC}{\mathbb{C}}   \newcommand{\bF}{\mathbb{F}}          
 \newcommand{\bR}{\mathbb{R}}        \newcommand{\bZ}{\mathbb{Z}}

\makeatletter
\let\c@equation\c@theorem
\makeatother
\numberwithin{equation}{section}

\newcommand{\sma}{\wedge} 

\newcommand{\TR}{\textnormal{TR}}
\newcommand{\TC}{\textnormal{TC}}

\begin{document}

\begin{frontmatter}

\title{$RO(S^1)$-graded $\TR$--groups of $\bF_p$, $\bZ$ and $\ell$}
\author{Vigleik Angeltveit\corref{cor1}}
\ead{vigleik@math.uchicago.edu}
\address{Department of Mathematics, University of Chicago, 5734 S University Ave, Chicago IL 60637} 
\author{Teena Gerhardt}
\cortext[cor1]{Partially supported by the NSF}
\ead{teena@math.msu.edu}
\address{A-218 Wells Hall, Michigan State University, East Lansing, MI 48824}


\begin{abstract}
We give an algorithm for calculating the $RO(S^1)$-graded $\TR$--groups of $\bF_p$, completing the calculation started by the second author. We also calculate the $RO(S^1)$-graded $\TR$--groups of $\bZ$ with mod $p$ coefficients and of the Adams summand $\ell$ of connective complex $K$-theory with $V(1)$-coefficients. Some of these calculations are used elsewhere to compute the algebraic $K$-theory of certain $\bZ$-algebras.
\end{abstract}

\end{frontmatter}

\section{Introduction}

Higher algebraic $K$-theory associates to a ring or ring spectrum $A$ a spectrum $K(A)$ and a sequence of abelian groups $K_i(A)$ which are the homotopy groups of this spectrum. Although higher algebraic $K$-theory was defined more than 30 years ago, computational progress has been slow. While the definition of algebraic $K$-theory is not inherently equivariant, the tools of equivariant stable homotopy theory have proven useful for $K$-theory computations via trace methods \cite{BHM}. The equivariant stable homotopy computations in this paper serve as input for these methods. In particular they have been used in the computations of the relative algebraic $K$-theory groups $K_*(\bZ[x]/(x^m),(x))$ and  $K_*(\bZ[x,y]/(xy),(x,y))$ up to extensions (see \cite{AnGeHe} and \cite{AnGe2} respectively).

The idea behind the trace methods is to approximate algebraic $K$-theory  by invariants of ring spectra which are more computable. The first approximation is topological Hochschild homology \cite{Bo1}, $T(A)$. This is significantly easier to compute than algebraic $K$-theory and there is a trace map $K(A) \rightarrow T(A)$ called the topological Dennis trace. A refinement of topological Hochschild homology called topological cyclic homology, TC$(A)$, serves as an even better approximation of algebraic $K$-theory. Indeed, there is a map $trc: K(A) \rightarrow \TC(A)$ called the cyclotomic trace \cite{BHM}  which is often close to an equivalence \cite{HeMa97, Mc97, GeHe06}. So in good cases trace methods reduce the computation of algebraic $K$-theory, $K_q(A)$, to that of topological cyclic homology, $\TC_q(A)$.

Topological cyclic homology is defined as a homotopy limit of certain fixed points of topological Hochschild homology. Let $p$ be a prime. The circle $S^1$ acts on $T(A)$ and we define $\TR^n(A;p)=T(A)^{C_{p^{n-1}}}$ to be the fixed point spectrum under the action of the cyclic group of order $p^{n-1}$ considered as a subgroup of $S^1$. It is important that $T(A)$ is a genuine $S^1$-equivariant spectrum, i.e., the spaces of $T(A)$ are indexed on a complete universe of $S^1$-representations. For a genuine $G$-spectrum $E$, the $H$-fixed point spectrum $E^H$ for $H \subset G$ has $n$'th space $E(\bR^n)^H$.


These spectra are connected by maps $R$, $F$, $V$ and $d$ \cite{HeMa05}, and a homotopy limit over $R$ and $F$ gives us the topological cyclic homology spectrum $\TC(A;p)$. Therefore to compute topological cyclic homology, and hence algebraic $K$-theory in good cases, it is sufficient to understand  $\TR^n(A;p)$ together with $R, F : \TR^{n+1}_*(A;p) \to \TR^n_*(A;p)$ for each $p$ and $n$. The homotopy groups of these spectra are denoted
$$
\TR^n_q(A;p) = [S^q \wedge S^1/C_{p^{n-1}+}, T(A)]_{S^1}.
$$
Throughout this paper the prime $p$ will be implicit. Hence we will write $\TR^n_q(A)$ for $\TR^n_q(A;p)$ and $\TC(A)$ for $\TC(A;p)$.

One type of singular ring for which the algebraic $K$-theory is particularly approachable is a pointed monoid algebra, $A(\Pi)$. This approach was first used by Hesselholt and Madsen \cite{HeMa97b} to compute the algebraic $K$-theory of $\mathbb{F}_p[x]/x^m$. To compute the $K$-theory of $A(\Pi)$ using the approach outlined above one first needs to understand the topological Hochschild homology $T(A(\Pi))$. Hesselholt and Madsen \cite{HeMa97} proved that there is an equivalence of $S^1$-spectra
\begin{equation} \label{eq:TAPisplitting}
T(A(\Pi)) \simeq T(A) \wedge B^{cy}(\Pi),
\end{equation}
where $B^{cy}(\Pi)$ denotes the cyclic bar construction on the pointed monoid $\Pi$. As above, trace methods essentially reduce the computation of $K_q(A(\Pi))$ to that of
$$
\TR^n_q(A(\Pi)) = \pi_q(T(A(\Pi))^{C_{p^{n-1}}}) = [S^q \wedge S^1/C_{p^{n-1}+}, T(A(\Pi))]_{S^1}.
$$
Using Equation \ref{eq:TAPisplitting} we can rewrite this as
$$
\TR^n_q(A(\Pi)) =[S^q \wedge S^1/C_{p^{n-1}+}, T(A) \wedge B^{cy}(\Pi)]_{S^1}.
$$
If one can understand how $B^{cy}(\Pi)$ is built out of $S^1$-representation spheres this gives a formula for these TR-groups in terms of groups of the form
$$
\TR^n_{q-\lambda}(A) =[S^q \wedge S^1/C_{p^{n-1}+}, T(A) \wedge S^{\lambda}]_{S^1}.
$$
Here $\lambda$ is a finite-dimensional $S^1$-representation and $S^{\lambda}$ denotes the one-point compactification of this representation. These groups are $RO(S^1)$-graded equivariant homotopy groups of the $S^1$-spectrum $T(A)$.  Recall that $RO(S^1)$ is the ring of virtual real representations of $S^1$, meaning that an element $\alpha \in RO(S^1)$ can be written as
\[ \alpha = [\beta]-[\gamma] \]
where $\beta$ and $\gamma$ are finite-dimensional real $S^1$-representations.  For $\alpha=[\beta]-[\gamma]$ in $RO(S^1)$ the $\TR$-group $\TR^n_\alpha(A)$ is defined by
\[ \TR^n_\alpha(A) = \pi_\alpha T(A)^{C_{p^{n-1}}} = [S^\beta \sma S^1/C_{p^{n-1}+}, S^\gamma \sma T(A)]_{S^1},\]
generalizing the integer-graded $\TR$-groups. As described above, these $RO(S^1)$-graded TR-groups arise naturally in the computation of the algebraic $K$-theory of some singular rings. Indeed, in some cases the computation of the algebraic $K$-theory groups $K_q(A(\Pi))$ can be reduced to the computation of the $RO(S^1)$-graded TR-groups $\TR^n_{q-\lambda}(A)$. However, few computations of these $RO(S^1)$-graded TR-groups have been done. The groups $\TR^n_{\alpha}(A)$ are only known in general when $A = \mathbb{F}_p$ and the dimension of $\alpha$ is even \cite{Ge07}. The current paper broadly extends what is known about $RO(S^1)$-graded TR-groups, making computations for $A = \mathbb{F}_p, \mathbb{Z},$ and $\ell$.

We use the results of this paper in \cite{AnGeHe}, which is joint work with Lars Hesselholt, to compute the relative $K$-groups $K_*(\bZ[x]/(x^m),(x))$ up to extensions, and in \cite{AnGe2} to compute the relative $K$-groups $K_*(\bZ[x,y]/(xy),(x,y))$ up to extensions. Theorem \ref{thm:mainL} below is the necessary input to the trace method approach described above, allowing us to make such computations. For example, we compute the relative $\TC$-groups $\TC_*(\bZ[x]/(x^m), (x); \bZ/p)$. Combined with a rational computation this tells us the rank and the number of torsion summands in each degree and in particular that
\[ \TC_{2i+1}(\bZ[x]/(x^m), (x)) \cong \bZ^{m-1} \]
is torsion free. An Euler characteristic argument then gives the order of the torsion groups.

The computations in this paper are also motivated by our interest in understanding the algebraic structure satisfied by the $RO(S^1)$-graded $\TR$-groups. The algebraic structure satisfied by the ordinary ($\bZ$-graded) $\TR$-groups is very rigid and this has proven quite useful \cite{HeMa97, HeMa03}, for example by considering the universal example. A better understanding of the algebraic structure of the $RO(S^1)$-graded $\TR$-groups should be similarly useful, and this is an area for further study. The computations in this paper provide important examples that we hope will be helpful in this regard.

Note that in cases where computing $\TR^n_*(A)$ with integral coefficients proves to be too difficult one can instead consider the groups $\TR^n_*(A;V)=\pi_*( T(A)^{C_{p^{n-1}}} \sma V)$ for a suitable finite complex $V$. For instance, smashing with the mod $p$ Moore spectrum $V(0)=S/p$ was used in \cite{BoMa} to compute the mod $p$ groups $\TR^n_*(\bZ;V(0)) = \TR^n_*(\bZ; \bZ/p)$ for $p \geq 3$. Similarly, smashing with the Smith-Toda complex $V(1) = S/(p,v_1)$ was used in \cite{AuRo} to compute $\TR^n_*(\ell; V(1))$ for $p \geq 5$. Here $\ell$ is the Adams summand of connective complex $K$-theory localized at $p$. In both of these cases, the $*$ refers to an integer grading. We will use this technique of smashing with a finite complex in our computations, which are $RO(S^1)$-graded.


In this paper we calculate $\TR^n_\alpha(\mathbb{F}_p)$, the $RO(S^1)$-graded TR-groups of $\mathbb{F}_p$,  $\TR^n_\alpha(\mathbb{Z}; V(0))$, the $RO(S^1)$-graded TR-groups of $\mathbb{Z}$ with mod $p$ coefficients, and  $\TR^n_\alpha(\ell; V(1))$, the $RO(S^1)$-graded TR-groups of $\ell$ with $V(1)$ coefficients. For the last case we assume $p \geq 5$, as $V(1)$ does not exist at $p=2$ and is not a ring spectrum at $p=3$. If $V$ is a ring spectrum, $\TR^n_{\alpha+*}(A; V)$ for fixed $\alpha$ will be a module over the integer-graded $\TR^n_*(A;V)$. While $V(0)$ is not a ring spectrum at $p=2$, our computation of $\TR^n_\alpha(\bZ; V(0))$ is still valid additively. This depends on a clever extension of the integer-graded computation of $\TR^n_*(\bZ;V(0))$ to the case $p=2$ that was carried out by Rognes in \cite{Ro99b}, using that $V(0)$ is a module over the mod $4$ Moore spectrum $S/4$.

The calculations in these three cases are essentially identical. To treat all three cases simultaneously, we introduce an integer $c \geq 0$, the chromatic level. If $c=0$ we let $A=\bF_p$ and use integral coefficients. If $c=1$ we let $A=\bZ$ and use mod $p$ coefficients. If $c=2$ we let $A=\ell$ and use $V(1)$-coefficients. Given a prime $p$ such that the spectrum $BP\langle c \rangle$ with homotopy groups $\bZ_{(p)}[v_1,\ldots,v_c]$ (or its $p$-completion) is $E_\infty$ and the Smith-Toda complex $V(c)$ exists and is a ring spectrum, the obvious generalization of the calculations in the paper applies.

In light of the problems with $V(1)$ mentioned above at $p=2$ and $p=3$, the following restriction on $p$ will be in force throughout the paper:
\begin{assumption}
If $c=0$ or $c=1$, $p$ can be any prime. If $c=2$, we assume $p \geq 5$.
\end{assumption}
The case $c=1$, $p=2$ is special, and in those arguments where we would normally use a ring structure (e.g.\ the proof of Theorem \ref{thm:hotoTRSS}) we have to instead use a module structure over the corresponding object with mod $4$ coefficients.

To state some of these results, we must first introduce some notation. Given a virtual real representation $\alpha \in RO(S^1)$, we define a prime operation by $\alpha'=\rho_p^* \alpha^{C_p}$ where $\rho_p : S^1 \to S^1/C_p$ is the isomorphism given by the $p$'th root \cite{Ge07}. We let $\alpha^{(n)}$ denote the $n$-fold iterated prime operation applied to $\alpha$. A real $S^1$-representation can be decomposed as a direct sum of copies of the trivial representation $\bR$ and the $2$-dimensional representations $\bC(n)$ with action given by $\lambda \cdot z = \lambda^n z$ for $n \geq 1$. The prime operation acts on these summands as follows:
\[ \bC(n)' =
\begin{cases}
\bC(\frac{n}{p}) & \text{ if $p \mid n$,} \\
0 & \text{otherwise.}
\end{cases} \]
and $\bR' = \bR$.

Given a virtual real representation $\mu$, we often write $\mu=\alpha+q$ as a sum of a complex representation $\alpha \in R(S^1)$ and a trivial representation $q \in \bZ$. Let $d_i(\alpha)=\dim_\bC(\alpha^{(i)})$. The $RO(S^1)$-graded $\TR$-groups considered in this paper all have the property that $\TR^n_{\alpha+*}$ for $* \in \bZ$ is determined by the sequence of integers
\[ d_0(\alpha), \ldots, d_{n-1}(\alpha).\]
Given any sequence of integers $d_0,\ldots,d_{n-1}$ it is possible to find a virtual representation $\alpha$ with $d_i=d_i(\alpha)$ for each $i$. Conversely, the $p$-homotopy type of $S^\alpha$ as a $C_{p^{n-1}}$-equivariant spectrum is determined by the integers $d_0(\alpha), \ldots, d_{n-1}(\alpha)$, so the fact that $\TR^n_{\alpha+*}$ is determined by these integers is perhaps not surprising. If $\alpha=\lambda$ or $\alpha=-\lambda$ for an actual representation $\lambda$, this sequence of integers is non-increasing or non-decreasing, respectively, and the TR-calculations simplify.

Fix an integer $c \in \{0, 1, 2\}$, and define
\begin{equation} \label{def:delta}
\delta_c^n(\alpha) = -d_0(\alpha)+\sum_{1 \leq k \leq n-1} \big[ d_{k-1}(\alpha)-d_k(\alpha) \big] p^{ck}.
\end{equation}
If $c=0$, let $A=\bF_p$ and $V=S^0$. If $c=1$, let $A=\bZ$ and $V=V(0)$. If $c=2$, let $A=\ell$ and $V=V(1)$.
We prove in Theorem \ref{thm:TateSS} below that in the stable range, i.e.\ for $q$ sufficiently large with respect to the integers $-d_i(\lambda)$, we have
\[ \TR^n_{\alpha+q}(A;V) \cong \TR^n_{q-2\delta^n_c(\alpha)}(A;V).\]
A similar result was obtained by Tsalidis \cite{Ts_pre} in the case $c=1$ for $\alpha=-\lambda$ where $\lambda$ is an actual $S^1$-representation.

We highlight the following result, which is essential to the $K$-theory computations in \cite{AnGeHe} and \cite{AnGe2}:
\begin{theorem} \label{thm:mainL}
Let $\lambda$ be a finite complex $S^1$-representation. Then for any prime $p$ the finite $\mathbb{Z}_{(p)}$-modules $\TR^n_{q - \lambda}(\mathbb{Z}; \mathbb{Z}/p)$ have the following structure:
\begin{enumerate}
\item For $ q\geq 2 d_0(\lambda), \TR^n_{q-\lambda}(\mathbb{Z}; \mathbb{Z}/p)$ has length $n$, if $q$ is congruent to $2\delta^n_1(\lambda)$ or $2\delta^n_1(\lambda) - 1$ modulo $2p^n$, and $n-1$ otherwise.
\item For  $2d_s(\lambda) \leq q < 2 d_{s-1}(\lambda)$ with $1 \leq s < n, \TR^n_{q - \lambda}(\mathbb{Z}; \mathbb{Z}/p)$ has length $n-s$ if $q$ is congruent to $2\delta^{n-s}_1(\lambda^{(s)})$ or $2 \delta^{n-s}_1(\lambda^{(s)})-1$ modulo $2p^{n-s}$ and $n-s-1$, otherwise.
\item For $q<2d_{n-1}(\lambda), \TR^n_{q-\lambda}(\mathbb{Z}; \mathbb{Z}/p)$ is zero.
\end{enumerate}
\end{theorem}

\noindent
At an odd prime $p$, $\TR^n_\alpha(\bZ;\bZ/p)$ is automatically a $\bZ/p$-vector space. It follows a posteriori that when $p=2$, $\TR^n_{q-\lambda}(\bZ;\bZ/2)$ is a $\bZ/2$-vector space; see \cite[Corollary 2.7]{AnGeHe}.

\subsection{Organization} \label{ss:organization}
We begin in \S \ref{Intro2} by recalling the fundamental diagram of TR-theory, which will be essential to the computations throughout the paper.  In \S \ref{s:hotoTRSS} we set up a spectral sequence from the homotopy groups of a homotopy orbit spectrum to the TR-groups we are aiming to compute. In \S \ref{s:TateSSc-1} we study the Tate spectral sequence in the $RO(S^1)$-graded setting, which is essential to understanding the homotopy orbit spectrum which serves as input for our computations.  We handle the cases of $\mathbb{F}_p$, $\mathbb{Z}$, and $\ell$ simultaneously. We find in Theorem \ref{thm:TateSS} below that in each case the Tate spectral sequence is a shifted version of the corresponding $\bZ$-graded spectral sequence. In \S \ref{s:hoorbVc-1} we study the effect of truncating the Tate spectral sequence to obtain spectral sequences converging to the homotopy orbits and the homotopy fixed points. This provides the induction step needed to prove Theorem \ref{thm:TateSS} from the previous section. In \S \ref{s:hotoTRSSVc-1} we describe the homotopy orbit to $\TR$ spectral sequence from \S \ref{s:hotoTRSS} in our examples for a general virtual representation $\alpha$. In \S \ref{s:TRFp} we consider the case $A=\bF_p$ and use the homotopy orbit to $\TR$ spectral sequence with $\bZ/p^l$ coefficients for all $l \geq 1$ to give an algorithm for computing $\TR^{n+1}_{\alpha+*}(\bF_p)$ for any virtual representation $\alpha$. In \S \ref{s:TRq-lambda} we specialize to representations of the form $-\lambda$, where $\lambda$ is an actual $S^1$-representation. We show that in this case the homotopy orbit to $\TR$ spectral sequence simplifies, and prove Theorem \ref{thm:mainL}.

\section{The fundamental diagram}\label{Intro2}

The $\TR$-groups are connected by several operators: $R$, $F$, $V$ and $d$. In the ordinary (integer-graded) case, there are maps as follows (see \cite{HeMa05} for more details). Inclusion of fixed points induces a map
$$ F : \TR^{n+1}_q(A) \to \TR^n_q(A) $$
called the Frobenius. This map has an associated transfer,
$$V : \TR^n_q(A) \to \TR^{n+1}_q(A), $$
the Verschiebung. The differential
$$ d : \TR^n_q(A) \to \TR^n_{q+1}(A) $$
is given by multiplying with the fundamental class of $S^1/C_{p^{n-1}}$ using the circle action. Topological Hochschild homology is a cyclotomic spectrum \cite{HeMa97}, which gives a map
\[ R : \TR^{n+1}_q(A) \to \TR^n_q(A) \]
called the restriction. The identification of the target of the restriction map with $\TR^n(A)$ uses this cyclotomic structure of $T(A)$, which identifies the geometric fixed points $T(A)^{gC_p}$ with $T(A)$. To make this identification we need to change universes, because the $S^1$ acting on $T(A)$ is not the same as the $S^1$ acting on $T(A)^{gC_p}$. As a special case, consider $T(G)$ for $G$ a topological group. Then $T(G) \simeq \Sigma^\infty Map(S^1, BG)_+$ is the suspension spectrum of the free loop space on $BG$. The geometric fixed points are then given by $T(G)^{gC_p} \simeq \Sigma^\infty Map(S^1/C_p,BG)_+$, the free loop space on loops parametrized by $S^1/C_p$.

The primary approach used to compute TR-groups is to compare the fixed point spectra to the homotopy fixed point spectra. Let $E$ denote a free contractible $S^1$-CW complex. Recall that the homotopy fixed point spectrum is defined by $T(A)^{hC_{p^n}} := F(E_+, T(A))^{C_{p^n}}$, and the TR-spectrum is defined by $\TR^{n+1}(A) := T(A)^{C_{p^n}}$. The map $E_+ \rightarrow S^0$ given by projection onto the non-basepoint induces a map
$$
\Gamma_n: \TR^{n+1}(A)  \rightarrow T(A)^{hC_{p^n}}.
$$
The general strategy for computing the homotopy groups $\TR^{n+1}_q(A)$ is to compute $\pi_q( T(A)^{hC_{p^n}})$ and the map $\Gamma_n$. This is facilitated through the use of a fundamental diagram of horizontal cofiber sequences, see \cite[\S 1-2]{BoMa} or \cite[Equation 25]{HeMa97}:
\begin{equation} \label{eq:NRD}
\xymatrix{ T(A)_{hC_{p^n}} \ar[r]^-N \ar[d]^= & \TR^{n+1}(A) \ar[r]^-R \ar[d]^{\Gamma_n} & \TR^n(A) \ar[r]^-\partial \ar[d]^{\hat{\Gamma}_n} & \Sigma T(A)_{hC_{p^n}} \ar[d]^= \\
T(A)_{hC_{p^n}} \ar[r]^-{N^h} & T(A)^{hC_{p^n}} \ar[r]^{R^h} & T(A)^{tC_{p^n}} \ar[r]^-\partial & \Sigma T(A)_{hC_{p^n}} }
\end{equation}
Let $\tilde{E}$ denote the cofiber of $E_+ \rightarrow S^0$. Then $T(A)_{hC_{p^n}} :=   (E_+ \wedge T(A))^{C_{p^{n}}}$ is the homotopy orbit spectrum and $T(A)^{tC_{p^n}} := (\tilde{E} \wedge F(E_+, T(A))^{C_{p^{n}}}$ is the Tate spectrum, see \cite{GrMa95}. A theorem of Tsalidis \cite[Theorem 2.4]{Ts98} characterizes situations when this map $\Gamma_n$ is an isomorphism.

The computation of $RO(S^1)$-graded TR-groups can be approached similarly. As before, we have the Frobenius $ F : \TR^{n+1}_\alpha(A) \to \TR^n_\alpha(A) $,  the Verschiebung $V : \TR^n_\alpha(A) \to \TR^{n+1}_\alpha(A)$, the differential $ d : \TR^n_\alpha(A) \to \TR^n_{\alpha+1}(A) $, and the restriction $R : \TR^{n+1}_\alpha(A) \to \TR^n_{\alpha'}(A).$ Note that the target of $R$ is the group in dimension $\alpha'$, not $\alpha$ (see \cite{HeMa97} for a detailed explanation of the restriction in this context).

The fundamental diagram also extends to this $RO(S^1)$-graded context. Let $T$ denote $T(A)$ and let $T[-\alpha]=T(A) \wedge S^{-\alpha}$ denote the desuspension of $T$ by $\alpha$. Then we have the following fundamental diagram of horizontal cofiber sequences, see \cite[Equation 49]{HeMa97}:

\begin{equation} \label{eq:NRD_equiv}
\xymatrix{ T[-\alpha]_{hC_{p^n}} \ar[r]^-N \ar[d]^= & \TR^{n+1}(A)[-\alpha] \ar[r]^-R \ar[d]^{\Gamma_n} & \TR^n(A)[-\alpha'] \ar[r]^-\partial \ar[d]^{\hat{\Gamma}_n} & \Sigma T[-\alpha]_{hC_{p^n}} \ar[d]^= \\
T[-\alpha]_{hC_{p^n}} \ar[r]^-{N^h} & T[-\alpha]^{hC_{p^n}} \ar[r]^{R^h} & T[-\alpha]^{tC_{p^n}} \ar[r]^-\partial & \Sigma T[-\alpha]_{hC_{p^n}} }
\end{equation}
Notice that $\TR^n(A)[-\alpha']$ appears, rather than $\TR^n(A)[-\alpha]$. We can take homotopy groups of the top row and get a long exact sequence
\begin{equation} \label{eq:fundLES}
\ldots \to \pi_q T[-\alpha]_{hC_{p^n}} \to \TR^{n+1}_{\alpha+q}(A) \to \TR^n_{\alpha'+q} (A)\to \pi_{q-1} T[-\alpha]_{hC_{p^n}} \to \ldots
\end{equation}
This is the fundamental long exact sequence of $RO(S^1)$-graded TR-theory. The strategy for computing $\TR^n_{\alpha+*}(A)$ is to use Diagram \ref{eq:NRD_equiv} and induction. One can attempt to understand the bottom row via spectral sequences, see \cite{GrMa95} and \cite[Equation 26]{HeMa97}. In this case the spectral sequences look as follows:

\begin{alignat*}{5}
\hat{E}^2_{s,t}(\alpha) & = \hat{H}^{-s}(C_{p^n}, V_t(T[-\alpha])) & \Rightarrow V_{s+t}( T[-\alpha]^{tC_{p^n}}) \\
E^2_{s,t}(\alpha) & = H^{-s}(C_{p^n}, V_t(T[-\alpha])) & \Rightarrow V_{s+t}( T[-\alpha]^{hC_{p^n}}) \\
E^2_{s,t}(\alpha) & = H_s(C_{p^n}, V_t(T[-\alpha])) & \Rightarrow V_{s+t}( T[-\alpha]_{hC_{p^n}})
\end{alignat*}

Note that in general we have $\hat{H}^k(C_{p^n}, M) \cong H^k(C_{p^n}, M)$ for $k > 0$ and $\hat{H}^k(C_{p^n}, M) \cong H_{-(k+1)}(C_{p^n}, M)$ for $k<-1$, and that when $M=\bZ/p$ we have $\hat{H}^0(C_{p^n}, M) \cong H^0(C_{p^n}, M)$ and $\hat{H}^{-1}(C_{p^n}, M) \cong H_0(C_{p^n}, M)$. This means that the restriction of the Tate spectral sequence to the first quadrant, meaning filtration $\geq 1$, gives the homotopy orbit spectral sequence with the filtration shifted by $1$. This corresponds to the connecting homomorphism $T[-\alpha]^{tC_{p^n}} \to \Sigma T[-\alpha]_{hC_{p^n}}$ in Diagram \ref{eq:NRD_equiv} above. Similarly, the restriction of the Tate spectral sequence to the second quadrant, meaning filtration $\leq 0$, gives the homotopy fixed point spectral sequence.

We use these spectral sequences to make computations of the homotopy groups on the bottom of the diagram. Understanding the maps $\Gamma_n$ and $\hat{\Gamma}_n$ is also key to our arguments.  Theorem \ref{t:Tsalidis} below, which is due to Tsalidis \cite{Ts98} in the non-equivariant case, says that if $\hat{\Gamma}_1$ is an isomorphism in sufficiently high degrees then so are $\Gamma_n$ and $\hat{\Gamma}_n$ for all $n$. If we know $\TR^n_{\alpha'+*}$ we can then use $\hat{\Gamma}_n$ to understand the Tate spectrum $T[-\alpha]^{tC_{p^n}}$ and the rest of the bottom row. This gives $\TR^{n+1}_{\alpha+*}$ in sufficiently high degrees. We are then left to compute $\TR^{n+1}_{\alpha+*}$ in the unstable range. In the following section we develop a spectral sequence that allows us to do the computations in the unstable range. This spectral sequence starts with the homotopy groups of various homotopy orbit spectra and converges to the $\TR$-groups we would like to compute. The spectral sequence allows us to treat the cases of $\mathbb{F}_p$, $\mathbb{Z}$, and $\ell$ simultaneously. However, in the case of $\mathbb{F}_p$ there are additional extension issues which need to be resolved.

In the $\bZ$-graded case, it is useful to first compute $\TR^n_*(\bF_p;\bZ/p)$. This shows that $\TR^n_{2q}(\bF_p)$ is cyclic and $\TR^n_{2q+1}(\bF_p)=0$, and from this we conclude that the relevant extensions are maximally nontrivial. In the $RO(S^1)$-graded case, $\TR^n_{\alpha+q}(\bF_p)$ could have several summands, and indeed, for many $\alpha$ it does. It is possible to compute the order of  $\TR^n_{\alpha+q}(\bF_p)$ inductively using Diagram \ref{eq:NRD_equiv}, and computations with $\mathbb{Z}/p$-coefficients determine the number of summands, but this information is not enough to determine the group. We solve this problem by using $\mathbb{Z}/p^l$ coefficients for all $l \geq 1$, calculating the associated graded of $\TR^n_{\alpha+q}(\bF_p;\bZ/p^l)$, and this is enough to solve the extension problem.

No such extension problems arise in our computations of $\TR^n_{\alpha+*}(\bZ;V(0))$ and $\TR^n_{\alpha+*}(\ell;V(1))$  as graded abelian groups. However, it is convenient to consider these not only as graded abelian groups but as modules over $\bF_p[v_1]$ using the map $v_1 : \Sigma^{2p-2} V(0) \to V(0)$ in the first case and over $\bF_p[v_2]$ using the map $v_2 : \Sigma^{2p^2-2} V(1) \to V(1)$ in the second case. This simplifies the bookkeeping, and by writing $\bZ/p^n$ as $\bF_p[v_0]/v_0^n$ we can treat all three cases simultaneously. In the stable range the module structure over $\bF_p[v_c]$ is clear, but there could be hidden $v_c$-multiplications in low degree. One could then consider using $S(p,v_1^l)$ or $S/(p,v_1,v_2^l)$ as coefficients, and although we believe this would give a similar algorithm for resolving the extensions as the one we find for $\TR^n_{\alpha+*}(\bF_p)$ we will not pursue that avenue here. We will express $\TR^n_{\alpha+*}(\bZ;V(0))$ as an $\bF_p[v_1]$-module and $\TR^n_{\alpha+*}(\ell;V(1))$ as an $\bF_p[v_2]$-module, with the caveat that there might be additional hidden extensions.

At $p=2$ there is no map $v_1 : \Sigma^2 V(0) \to V(0)$, so it does not make sense to express $\TR^n_{\alpha+*}(\bZ, V(0))$ as a $\bF_2[v_1]$-module. So when we write down $\TR^n_{\alpha+*}(\bZ; V(0))$ the result should be interpreted additively, or as a module over $\bF_2[v_1^4]$ using the map $v_1^4 : \Sigma^8 V(0) \to V(0)$, when $p=2$.

\section{The homotopy orbit to $\TR$ spectral sequence} \label{s:hotoTRSS}
It is possible to glue together the long exact sequences in Equation \ref{eq:fundLES} to obtain a spectral sequence converging to $\TR^{n+1}_{\alpha+*}(A;V)$ with coefficients in $V$. For this section $A$ can be any connective $S$-algebra and $V$ can be any spectrum. Let $T=T(A)$. The $E^1$ term is given by
\[ E^1_{s,t}(\alpha)=\begin{cases} V_t T[-\alpha^{(n-s)}]_{hC_{p^s}} & \text{for $0 \leq s \leq n$,} \\ 0 & \text{otherwise.} \end{cases} \]
This spectral sequence converges to $\TR^{n+1}_{\alpha+t}(A;V)$. Note that we use a slightly non-standard grading convention here; we find it more convenient not to deal with $V_{t-s} T[-\alpha^{(n-s)}]_{hC_{p^s}}$.

The reason this spectral sequence has not been introduced before is that in previously computed examples, one can understand $\TR^{n+1}_*(A;V)$ completely by comparing with $V_* T^{hC_{p^n}}$. In the $RO(S^1)$-graded case, there is a range of degrees where this comparison is less useful.

The $d_r$ differential has bidegree $(r,-1)$,
\[ d_r : E^r_{s,t}(\alpha) \to E^r_{s+r,t-1}(\alpha),\]
and can be defined as follows: For $x \in V_t T[-\alpha^{(n-s)}]_{hC_{p^s}}$, $d_r(x)$ is given by lifting $N(x)$ up to $\TR^{s+r}_{\alpha^{(n-s-r+1)}+t}(A;V)$ and then applying $\partial$:
\[ \xymatrix{ & \TR^{s+r}_{\alpha^{(n-s-r+1)}+t}(A;V) \ar[r]^-\partial \ar[d]^R & V_{t-1} T[-\alpha^{(n-s-r)}]_{hC_{p^{s+r}}} \\
 & \vdots \ar[d]^R & \\
 & \TR^{s+2}_{\alpha^{(n-s-1)}+t}(A;V) \ar[d]^-R \ar[r]^-\partial & V_{t-1} T[-\alpha^{(n-s-2)}]_{hC_{p^{s+2}}} \\
V_t T[-\alpha^{(n-s)}]_{hC_{p^s}} \ar[r]^-N & \TR^{s+1}_{\alpha^{(n-s)}+t}(A;V) \ar[r]^-\partial & V_{t-1} T[-\alpha^{(n-s-1)}]_{hC_{p^{s+1}}}
} \]

\begin{observation}
We note that if $A$ and $V$ are $(-1)$-connected the filtration $s$ piece $V_* T[-\alpha^{(n-s)}]_{hC_{p^s}}$ is zero in degree $*<-2d_{n-s}(\alpha)$.
\end{observation}

\begin{defn} \label{def:Tatepiece}
Consider the short exact sequence
\[ 0 \to coker(R^h)[-1] \to V_* T[-\alpha]_{hC_{p^n}} \to ker(R^h) \to 0 \]
obtained by taking $V_*(-)$ of the bottom row of Diagram \ref{eq:NRD_equiv}. We call the image of $coker(R^h)[-1]$ in $V_* T[-\alpha]_{hC_{p^n}}$ the Tate piece and denote it by $V^t_* T[-\alpha]_{hC_{p^n}}$. If the sequence is split we choose a splitting and call the image of $ker(R^h)$ under the splitting the homotopy fixed point piece, denoted $V^h_* T[-\alpha]_{hC_{p^n}}$.

Hence if the above short exact sequence splits we have a decomposition
\[ V_* T[-\alpha]_{hC_{p^n}} \cong V^t_* T[-\alpha]_{hC_{p^n}} \oplus V^h_* T[-\alpha]_{hC_{p^n}}.\]
\end{defn}

The purpose of the above definition is to get a better handle on the differentials in the homotopy orbit to TR spectral sequence:

\begin{lemma}
In the homotopy orbit to TR spectral sequence, every class in the Tate piece $V^t_* T[-\alpha^{(n-s)}]_{hC_{p^s}}$ is a permanent cycle, and the image of any differential is contained in the Tate piece. If the short exact sequence in Definition \ref{def:Tatepiece} splits then all differentials go from a subgroup of the homotopy fixed point piece to a quotient of the Tate piece.
\end{lemma}

\begin{proof}
This is a straightforward diagram chase, using the construction of the spectral sequence and Diagram \ref{eq:NRD_equiv}.
\end{proof}

We will denote classes in $V^t_*T[-\alpha]_{hC_{p^n}}$ by their name in $V_*T[-\alpha]^{tC_{p^n}}$ and classes in $V^h_*T[-\alpha]_{hC_{p^n}}$ by their name in $V_*T[-\alpha]^{hC_{p^n}}$.

\section{The Tate spectral sequence} \label{s:TateSSc-1}
In order to use the spectral sequence from the previous section, we must first understand the homotopy orbit spectrum. The homotopy orbit spectral sequence computing $V_* T[-\alpha]_{hC_{p^n}}$ is the restriction of the corresponding Tate spectral sequence to positive filtration, so we first need to study the Tate spectral sequence converging to $V_* T[-\alpha]^{tC_{p^n}}$. For $c \in \{0,1,2\}$, let $V$ and $A$ be as in the introduction, and let $T=T(A)$.

Recall \cite{HeMa97, BoMa, AuRo} that the homotopy groups of topological Hochschild homology with these coefficients are given as follows:

\begin{eqnarray*}
\pi_* T(\bF_p) & = & P(\mu_0), \\
V(0)_* T(\bZ) & = & E(\lambda_1) \otimes P(\mu_1), \\
V(1)_* T(\ell) & = & E(\lambda_1,\lambda_2) \otimes P(\mu_2).
\end{eqnarray*}
Here $P(-)$ denotes a polynomial algebra and $E(-)$ denotes an exterior algebra, both over $\bF_p$. The degrees are given by $|\lambda_i|=2p^i-1$ and $|\mu_c|=2p^c$, with $\lambda_i$ represented by $\sigma \bar{\xi}_i$ and $\mu_c$ represented by $\sigma \bar{\tau}_c$ in the B\"okstedt spectral sequence. At $p=2$, $\lambda_i$ is represented by $\sigma \bar{\xi}_i^2$ and $\mu_c$ is represented by $\sigma \bar{\xi}_{c+1}$.

The above formula for $V(0)_* T(\bZ)$ can be interpreted multiplicatively even though $V(0)$ is not a ring spectrum at $p=2$, by using that $V(0) \sma T(\bZ) \simeq T(\bZ;\bZ/2)$, topological Hochschild homology of $\bZ$ with coefficients in the bimodule $\bZ/2$. (A similar trick gives an interpretation of $V(1)_* T(\ell)$ at $p=2$ and $p=3$, but we will not need this.) But note that there is no $S^1$-action on topological Hochschild homology with coefficients in a bimodule, so there is no corresponding ring structure on the $\TR$-groups if the coefficient spectrum is not a ring spectrum. Rognes \cite{Ro99b} has shown that at $p=2$ everything still works, by showing the Tate spectral sequence converging to $V(0)_* T(\bZ)^{tC_{p^n}}$ has a \emph{formal} algebra structure, so we can proceed as if $V(0) \sma T(\bZ)^{tC_{p^n}}$ was a ring spectrum.


We have
\[ V_* T[-\alpha] \cong V_{2d_0(\alpha)+*} T,\]
and we know from \cite[Lemma 9.1]{HeMa97} that the Tate spectrum $T[-\alpha]^{tC_{p^n}}$ only depends on $\alpha'$. With the usual grading conventions the Tate spectral sequence will depend on $\alpha$, and not just on $\alpha'$. In fact, by considering the Tate spectral sequence for some $\beta$ with $\alpha'=\beta'$ the pattern of differentials will change in the following way. If we have a differential
\[ d_{2r}(t^k x) = t^{k+r} y \]
in the spectral sequence converging to $V_* T[-\alpha]^{tC_{p^n}}$ we get a differential
\[ d_r(t^{k-d_0(\beta)+d_0(\alpha)} x) = t^{k-d_0(\beta)+d_0(\alpha)+r} y \]
in the spectral sequence converging to $V_* T[-\beta]^{tC_{p^n}}$.

To get a Tate spectral sequence that only depends on $\alpha'$, we do the following. Write
\begin{equation} \label{eq:ValphaT}
V_* T[-\alpha] = t^{d_0(\alpha)} V_* T
\end{equation}
where $|t|=-2$.

Then the Tate spectral sequence converging to $V_* T[-\alpha]^{tC_{p^n}}$ has $E_2$ term given by
\begin{equation*}
\hat{E}_2(\alpha) = \hat{H}^*(C_{p^n}; V_* T[-\alpha]) \cong V_* T(A) \otimes P(t,t^{-1}) \otimes E(u_n)[-\alpha],
\end{equation*}
a free module over the corresponding non-equivariant spectral sequence on a generator $[-\alpha]$. Here $|u_n|=-1$ and $|t|=-2$ are in negative filtration degree ($s$) and zero fiber degree ($t$), while $V_* T(A)$ is concentrated in filtration degree $0$. With a factor of $t^{d_0(\alpha)}$ coming from $V_* T[-\alpha]$, the Tate spectral sequence now only depends on $\alpha'$ and the $E_2$ term is isomorphic as a bigraded abelian group to the corresponding non-equivariant $E_2$ term. The price we pay is that we have to redefine what we mean by the first and second quadrant of this spectral sequence. Now first quadrant means filtration $\geq -2d_0(\alpha)+1$ and second quadrant means filtration $\leq -2d_0(\alpha)$.

The class $v_c \in \pi_{2p^c-2} V$ (recall that $v_0=p$) maps to a class in $V_* T^{hS^1}$ represented by $t \mu_c$ in the $E_2$ term of the homotopy fixed point spectral sequence (see e.g.\ \cite[Proposition 4.8]{AuRo}), so by abuse of notation we will denote the class $t \mu_c$ in the $C_{p^n}$ Tate spectral sequence by $v_c$.

Recall \cite{HeMa97, BoMa, AuRo} that $V_* T^{tC_{p^n}}$ is $2p^{cn}$-periodic and the definition of $\delta_c^n(\alpha)$ in Equation \ref{def:delta} in the introduction.

\begin{thm} \label{thm:TateSS}
The $RO(S^1)$-graded $\TR$ groups of $A$ satisfy
\[ \TR^n_{\alpha+*}(A;V) \cong \TR^n_{*-2\delta_c^n(\alpha)}(A;V) \]
for $*$ sufficiently large, and the $V$-homotopy groups of $T[-\alpha]^{tC_{p^n}}$ satisfy
\[ V_*T[-\alpha]^{tC_{p^n}} \cong V_{*-2\delta_c^n(\alpha')} T^{tC_{p^n}} \]
for all $*$.
\end{thm}

We prove this theorem in the next section, after analyzing the restriction of the Tate spectral sequence to the first and second quadrant. The proof goes by induction, using a version of Tsalidis' theorem (Theorem \ref{t:Tsalidis}). The point is that knowing $\TR^n_{\alpha'+*}(A;V)$ in the stable range tells us about the behavior of the Tate spectral sequence converging to $V_{\alpha+*} T(A)^{tC_{p^n}}$, which by restriction to the second quadrant tell us about $V_{\alpha+*} T(A)^{hC_{p^n}}$ and hence about $\TR^{n+1}_\alpha(A; V)$.

We spell out the behavior of the Tate spectral sequence in each case. The proof of Theorem \ref{thm:TateSS}, as well as the following formulas, are proved after Theorem \ref{t:Tsalidis} in the next section. Define $r(n)$ by
\begin{equation} \label{def:r(n)}
r(n) = \sum_{1 \leq k \leq n} p^{ck}.
\end{equation}
As in the non-equivariant case the classes $\lambda_i$ and $v_c$ are permanent cycles, and the Tate spectral sequence is determined by the following (compare \cite{HeMa97, BoMa, AuRo}):

In each case we have a family of differentials given by
\[ d_{2r(n)+1}(t^{-k} u_n[-\alpha])=v_c^{r(n-1)+1} t^{p^{cn}-k} [-\alpha] \]
if $\nu_p(k-\delta_c^n(\alpha')) \geq cn$. If $c=0$ this condition is empty, and this is the only family of differentials.

For $c \geq 1$ we have, for each $1 \leq j \leq n$, a differential
\[ d_{2r(j)}(t^{-k}[-\alpha])=v_c^{r(j-1)} t^{p^{cj}-k} \lambda_c [-\alpha] \]
if $\nu_p(k-\delta_c^n(\alpha'))=cj-1$.

Finally, if $c=2$ we have, for each $1 \leq j \leq n$, a differential
\[ d_{2r(j)/p}(t^{-k}[-\alpha])=v_2^{r(j-1)/p} t^{p^{2j-1}-k} \lambda_1 [-\alpha] \]
if $\nu_p(k-\delta_2^n(\alpha'))=2j-2$.

\section{The homotopy orbit and homotopy fixed point spectra} \label{s:hoorbVc-1}
To find $V_*T[-\alpha]_{hC_{p^n}}$ and $V_*T[-\alpha]^{hC_{p^n}}$ we restrict the Tate spectral sequence from the previous section to the first or second quadrant. Recall that because of our grading conventions, in particular Equation \ref{eq:ValphaT} above, the first quadrant means filtration greater than $-2d_0(\alpha)$. Hence the homotopy orbit spectral sequence has $E_2$-term
\[ V_* T \otimes E(u_n)\{t^k [-\alpha] \, : \, k < d_0(\alpha) \}[-1] \]
and the homotopy fixed point spectral sequence has $E_2$-term
\[ V_* T \otimes E(u_n)\{t^k[-\alpha] \, : \, k \geq d_0(\alpha) \}.\]

Analyzing these spectral sequences is straightforward, but requires some amount of bookkeeping. We will write down $V_*T[-\alpha]_{hC_{p^n}}$ completely because it is the input to the homotopy orbit to TR spectral sequence. We will partially describe $V_*T[-\alpha]^{hC_{p^n}}$ by explaining how some $v_c$-towers in the homotopy fixed point piece of $V_*T[-\alpha]_{hC_{p^n}}$ become divisible by some power of $v_c$ in $V_*T[-\alpha]^{hC_{p^n}}$. The rest of $V_* T[-\alpha]^{hC_{p^n}}$ consists of those $v_c$-towers that are concentrated in negative total degree, and these are isomorphic to the corresponding $v_c$-towers in $V_* T[-\alpha]^{tC_{p^n}}$.

We separate $V_*T[-\alpha]_{hC_{p^n}}$ into the Tate piece and the homotopy fixed point piece as in Definition \ref{def:Tatepiece}, and each piece comes in $c+1$ families, each of which can be split into a stable part and an unstable part. In sufficiently high degrees the map $R^h$ in Diagram \ref{eq:NRD_equiv} is zero, so $N^h$ is an isomorphism between the homotopy fixed point piece of $V_*T[-\alpha]_{hC_{p^n}}$ and $V_*T[-\alpha]^{hC_{p^n}}$ in the stable range. This isomorphism can be described in terms of those differentials in the Tate spectral sequence which go from the first to the second quadrant. Such a differential leaves one class in $V_*T[-\alpha]_{hC_{p^n}}$ and one class in $V_*T[-\alpha]^{hC_{p^n}}$, neither of which has a corresponding class in $V_*T[-\alpha]^{tC_{p^n}}$.

To describe the first family, which is the one ``created'' by the longest differential $d_{2r(n)+1}$ in the Tate spectral sequence, let $E=\bF_p$ for $c=0$, $E(\lambda_1)$ for $c=1$ and $E(\lambda_1,\lambda_2)$ for $c=2$. Then the Tate piece of the first family splits as the following direct sum:
\begin{eqnarray*}
& \displaystyle \bigoplus_{\stackrel{k \geq r(n-1)+1-d_0(\alpha)}{\nu_p(k-\delta_c^n(\alpha')) \geq cn}} & E \otimes P_{r(n-1)+1}(v_c)\{t^{-k}[-\alpha]\}[-1] \\
& \displaystyle \bigoplus_{\stackrel{1 \leq k+d_0(\alpha) \leq r(n-1)}{\nu_p(k-\delta_c^n(\alpha')) \geq cn}} & E \otimes  P_{k+d_0(\alpha)}(v_c)\{t^{-k}[-\alpha]\}[-1]
\end{eqnarray*}
In particular, in the stable range we have $v_c$-towers of height $r(n-1)+1$ starting in degree
\[ 2\delta_c^n(\alpha')+mp^{cn}.\]

Similarly, the homotopy fixed point piece splits as a direct sum as follows:
\begin{eqnarray*}
& \displaystyle \bigoplus_{\stackrel{k \geq r(n-1)+1}{\nu_p(k-d_0(\alpha)-\delta_c^n(\alpha')) \geq cn}} & E \otimes P_{r(n)+1}(v_c)\{t^{d_0(\alpha)} \mu_c^k[-\alpha]\} \\
& \displaystyle \bigoplus_{\stackrel{1 \leq k \leq r(n)}{\nu_p(k-d_0(\alpha)-\delta_c^n(\alpha')) \geq cn}} & E \otimes P_k(v_c)\{v_c^{r(n)+1-k} t^{d_0(\alpha)} \mu_c^{k-p^{cn}}[-\alpha]\}
\end{eqnarray*}
In particular, in the stable range we have $v_c$-towers of height $r(n)+1$ starting in degree
\[ -2d_0(\alpha)+2p^c(d_0(\alpha)+\delta_c^n(\alpha')+mp^{cn}) = 2\delta_c^{n+1}(\alpha)+mp^{c(n+1)}.\]

Next we compare this to $V_* T[-\alpha]^{hC_{p^n}}$. For the $v_c$-towers of maximal height, the map $N^h$ in Equation \ref{eq:NRD_equiv} is an isomorphism. Now consider a generator $x$ of
\[ P_k(v_c)\{v_c^{r(n)+1-k} t^{d_0(\alpha)} \mu_c^{k-p^{cn}} [-\alpha]\} = P_k(v_c)\{t^{r(n)+1-k+d_0(\alpha)} \mu_c^{r(n-1)+1}[-\alpha]\} \]
and its image $N^h(x)$ in $V_*T[-\alpha]^{hC_{p^n}}$. We have two cases, with the first case only applicable if $c \geq 1$. First, if $k<p^{cn}$ then $N^h(x)$ is divisible by $v_c^{r(n-1)+1}$ and we get a $v_c$-tower
\[ E \otimes P_{r(n-1)+1+k}(v_c)\{t^{p^{cn}-k+d_0(\alpha)} [-\alpha]\}.\]
If $k \geq p^{cn}$ then $N^h(x)$ is divisible by $v_c^{r(n)+1-k}$ and we get a $v_c$-tower
\[ E \otimes P_{r(n)+1}(v_c)\{t^{d_0(\alpha)} \mu_c^{k-p^{cn}}[-\alpha]\}.\]

If $c \geq 1$ the second family is ``created'' by the differentials $d_{2r(j)}$ for $1 \leq j \leq n$. Let $E'_n=E(u_n)$ if $c=1$ and $E(\lambda_1,u_n)$ if $c=2$. Then the Tate piece of the second family splits as the following direct sum:
\begin{eqnarray*}
& \displaystyle \bigoplus_{2 \leq j \leq n} \bigoplus_{\stackrel{k \geq r(j-1)-d_0(\alpha)}{\nu_p(k-\delta_c^n(\alpha')) = cj-1}} & E_n' \otimes P_{r(j-1)}(v_c)\{t^{-k} \lambda_c [-\alpha]\}[-1] \\
& \displaystyle \bigoplus_{2 \leq j \leq n} \bigoplus_{\stackrel{1\leq k+d_0(\alpha) \leq r(j-1)-1}{\nu_p(k-\delta_c^n(\alpha')) = cj-1}} & E_n' \otimes P_{k+d_0(\alpha)}(v_c)\{t^{-k} \lambda_c [-\alpha]\}[-1]
\end{eqnarray*}
Similarly, the homotopy fixed point piece splits as a direct sum as follows:
\begin{eqnarray*}
& \displaystyle \bigoplus_{1 \leq j \leq n} \bigoplus_{\stackrel{k \geq r(j-1)}{\nu_p(k-d_0(\alpha)-\delta_c^n(\alpha')) = cj-1}} & E_n' \otimes P_{r(j)}(v_c)\{t^{d_0(\alpha)} \mu_c^k \lambda_c [-\alpha]\} \\
& \displaystyle \bigoplus_{1 \leq j \leq n} \bigoplus_{\stackrel{1 \leq k \leq r(j)-1}{\nu_p(k-d_0(\alpha)-\delta_c^n(\alpha')) = cj-1}} & E_n' \otimes P_k(v_c)\{v_c^{r(j)-k} t^{d_0(\alpha)} \mu_c^{k-p^{cj}} \lambda_c [-\alpha]\}
\end{eqnarray*}

Consider a generator $x$ of $P_k(v_c)\{v_c^{r(j)-k} t^{d_0(\alpha)} \mu_c^{k-p^{cj}} \lambda_c [-\alpha]\}$ and its image $N^h(x)$ in $V_*T[-\alpha]^{hC_{p^n}}$. Again we have two cases. If $k<p^{cj}$ then $N^h(x)$ is divisible by $v_c^{r(j-1)}$ and we get a $v_c$-tower
\[ E_n' \otimes P_{r(j-1)+k}(v_c)\{t^{p^{cj}-k+d_0(\alpha)} \lambda_c [-\alpha]\}. \]
If $k \geq p^{cj}$ then $N^h(x)$is divisible by $v_c^{r(j)-k}$ and we get a $v_c$-tower
\[ E_n' \otimes P_{r(j)}(v_c)\{t^{d_0(\alpha)} \mu_c^{k-p^{cj}} \lambda_c [-\alpha]\}.\]

Finally, if $c=2$ the third family is ``created'' by the differentials $d_{2r(j)/p}$ for $1 \leq j \leq n$. Let $E''_n=E(\lambda_2,u_n)$. Then the Tate piece of the third family splits as the following direct sum:
\begin{eqnarray*}
& \displaystyle \bigoplus_{2 \leq j \leq n} \bigoplus_{\stackrel{k \geq r(j-1)/p-d_0(\alpha)}{\nu_p(k-\delta_2^n(\alpha')) = 2j-2}} & E_n'' \otimes P_{r(j-1)/p}(v_2)\{t^{-k} \lambda_1 [-\alpha]\}[-1] \\
& \displaystyle \bigoplus_{2 \leq j \leq n} \bigoplus_{\stackrel{1 \leq k+d_0(\alpha) \leq r(j-1)/p-1}{\nu_p(k-\delta_2^n(\alpha')) = 2j-2}} & E_n'' \otimes P_{k+d_0(\alpha)}(v_2)\{t^{-k} \lambda_1 [-\alpha]\}[-1]
\end{eqnarray*}
Similarly, the homotopy fixed point piece splits as a direct sum as follows:
\begin{eqnarray*}
& \displaystyle \bigoplus_{1 \leq j \leq n} \bigoplus_{\stackrel{k \geq r(j-1)/p}{\nu_p(k-d_0(\alpha)-\delta_2^n(\alpha')) = 2j-2}} & E_n'' \otimes P_{r(j)/p}(v_2)\{t^{d_0(\alpha)} \mu_2^k \lambda_1 [-\alpha]\} \\
& \displaystyle \bigoplus_{1 \leq j \leq n} \bigoplus_{\stackrel{1 \leq k \leq r(j)/p-1}{\nu_p(k-d_0(\alpha)-\delta_2^n(\alpha')) = 2j-2}} & E_n'' \otimes P_k(v_2)\{v_2^{r(j)/p-k} t^{d_0(\alpha)} \mu_2^{k-p^{2j-1}} \lambda_1 [-\alpha]\}
\end{eqnarray*}

Once again, consider the image $N^h(x)$ of a generator $x$ of the $v_2$-tower $P_k(v_2)\{v_2^{r(j)/p-k} t^{d_0(\alpha)} \mu_2^{k-p^{2j-1}} \lambda_1 [-\alpha] \}$ in $V_*T[-\alpha]^{hC_{p^n}}$. If $k<p^{2j-1}$ then $N^h(x)$ is divisible by $v_2^{r(j-1)/p}$ and we get a $v_2$-tower
\[ E_n'' \otimes P_{r(j-1)/p+k}(v_2) \{ t^{p^{2j-1}-k+d_0(\alpha)} \lambda_1 [-\alpha] \}.\]
If $k \geq p^{2j-1}$ then $N^h(x)$ is divisible by $v_2^{r(j)/p-k}$ and we get a $v_2$-tower
\[ P_{r(j)/p}(v_2)\{t^{d_0(\alpha)} \mu_2^{k-p^{2j-1}} \lambda_1 [-\alpha] \}.\]

We will use the following theorem, which with integral coefficients is due to Tsalidis \cite[Theorem 2.4]{Ts98} in the $\bZ$-graded case and Hesselholt-Madsen \cite[Addendum 9.1]{HeMa97} in a special case of the $RO(S^1)$-graded case:
\begin{thm} \label{t:Tsalidis}
Let $A$ be a connective ring spectrum of finite type. Suppose the map $\hat{\Gamma}_1 : T(A) \to T(A)^{tC_p}$ induces an isomorphism $\pi_q (T(A); V) \to \pi_q (T(A)^{tC_p}; V)$ for $q \geq i$. Then, for any $n \geq 1$, $\hat{\Gamma}_n$ induces an isomorphism $\TR^n_{\alpha'+q}(A;V) \to V_qT[-\alpha]^{tC_{p^n}}$ for
\[ q \geq 2\max(-d_1(\alpha),\ldots,-d_n(\alpha))+i.\]
Equivalently, $\Gamma_n$ induces an isomorphism $\TR^{n+1}_{\alpha+q}(A;V) \to \pi_q(T[-\alpha]^{hC_{p^n}}; V)$ in the same range.
\end{thm}

\begin{proof}
The proof in \cite[Addendum 9.1]{HeMa97} goes through verbatim with $T(\bF_p)$ replaced by $V(c) \sma T(A)$.
\end{proof}

\begin{proof}[Proof of Theorem \ref{thm:TateSS}]
In each case Theorem \ref{t:Tsalidis} applies, see e.g.\ \cite[Proposition 5.3]{HeMa97} for $\bF_p$, \cite[Lemma 6.5]{BoMa} for $\bZ$, and \cite[Theorem 5.5]{AuRo} for $\ell$. For $c=0$ we have $i=0$, for $c=1$ we have $i=0$ and for $c=2$ we have $i=2p-1$ (the class $t^{p^2} \lambda_1 \lambda_2$ in $V(1)_* T(\ell)^{tC_p}$ is in degree $2p-2$). Suppose by induction that the statement of the Theorem holds for $\TR^n_{\alpha'+*}(A;V)$. Then the map $\hat{\Gamma}_n : \TR^n_{\alpha'+*}(A;V) \to V_*T[-\alpha]^{tC_{p^n}}$ is coconnective, so $V_*T[-\alpha]^{tC_{p^n}}$ is shifted by $2\delta_c^n(\alpha')$ degrees in the stable range. Using that $V_*T[-\alpha]^{tC_{p^n}}$ is a module over $V_* T^{tC_{p^n}}$ and that $V_* T^{tC_{p^n}}$ is $2p^n$-periodic the statement for $V_*T[-\alpha]^{tC_{p^n}}$ follows. The pattern of differentials in the Tate spectral sequence described after the statement of Theorem \ref{thm:TateSS} also follows from this.

Restricting the Tate spectral sequence to the second quadrant gives a spectral sequence computing $V_*T[-\alpha]^{hC_{p^n}}$, and each differential on a class $t^{-k}$ in the Tate spectral sequence gives a class
\[ t^{d_0(\alpha)} \mu_c^{k+d_0(\alpha)}.\]
The differentials on $t^{-k}$ for various $k$ are shifted by $2\delta_c^n(\alpha')$ degrees, which means that the classes in the homotopy fixed point spectrum are shifted by
\[ -2d_0(\alpha)+2p^c(d_0(\alpha)+\delta_c^n(\alpha')) = 2\delta_c^{n+1}(\alpha) \]
degrees. Using that $\Gamma_n : \TR^{n+1}_{\alpha+*}(A;V) \to V_*T[-\alpha]^{hC_{p^n}}$ is coconnective the statement then holds for $\TR^{n+1}_{\alpha+*}(A;V)$.
\end{proof}

\section{A splitting of the homotopy orbit to TR spectral sequence} \label{s:hotoTRSSVc-1}
In this section we describe the homotopy orbit to TR spectral sequence in the three cases of interest. We show that the spectral sequence splits as the direct sum of ``small'' spectral sequences, with no differentials between different summands.

We first describe the small spectral sequences. Consider the following diagram:

\small{
\[ \xymatrix{
& P_{r(0)+1}(v_c)\{t^{d_n(\alpha)}\mu_c^k\} \ar@{-}[d] \ar@{.>}[ld] \ar@{.>}[ldd] \ar@{.>}[ldddd] \\
P_{r(0)+1}(v_c)\{t^{-p^c k-\delta_c^1(\alpha^{(n)})}\}[-1] \ar@{-}[d] & \!\!\!\!\! P_{r(1)+1}(v_c)\{t^{d_{n-1}(\alpha)} \mu_c^{p^c k+d_{n-1}(\alpha)+\delta_c^1(\alpha^{(n)})}\} \ar@{-}[d] \ar@{.>}[ld] \ar@{.>}[lddd] \\
P_{r(1)+1}(v_c)\{t^{-p^{2c} k-\delta_c^2(\alpha^{(n-1)})}\}[-1] \ar@{-}[d] & \!\!\!\!\! P_{r(2)+1}(v_c)\{t^{d_{n-2}(\alpha)}\mu_c^{p^{2c} k+d_{n-2}(\alpha)+\delta_c^2(\alpha^{(n-1)})}\} \ar@{-}[d] \ar@{.>}[ldd] \\
\vdots \ar@{-}[d] & \vdots \ar@{-}[d] \\
P_{r(n-1)+1}(v_c)\{t^{-p^{cn} k-\delta_c^n(\alpha')}\}[-1] & P_{r(n)+1}(v_c)\{t^{d_0(\alpha)}\mu_c^{p^{cn} k+d_0(\alpha)+\delta_c^n(\alpha')}\}
} \]
}

\normalsize

For each $k$, there is a summand of the $E^1$ term of the homotopy orbit to TR spectral sequence which looks like the above diagram tensored with $E$ (recall that $E=\bF_p$, $E(\lambda_1)$ or $E(\lambda_1,\lambda_2)$), with submodules of the modules in the right hand column and quotient modules of the modules in the left hand column (the summands are allowed to be $0$). If $c=0$, this describes the whole $E^1$ term. If $c=1$ there is one more family of diagrams to consider and if $c=2$ there are two more families of diagrams to consider.

For $c=1$ or $2$ the second family of small spectral sequences looks as follows. Recall that $E_j'=E(u_j)$ if $c=1$ and $E(u_j,\lambda_1)$ if $c=2$. For each $0 \leq j \leq n-1$ and each $k$ with $\nu_p(k-d_j(\alpha)+d_{j+1}(\alpha)) = c-1$ we have a corresponding diagram, where the right hand side consists of submodules of
\[ E'_{n-j+m} \otimes P_{r(m+1)} (v_c)\{t^{d_{j-m}(\alpha)}  \mu_c^{p^{cm} k+d_{j-m}(\alpha)+\delta_c^m(\alpha^{(j-m+1)})} \lambda_c \} \]
for $0 \leq m \leq j$ and the left hand side consists of quotient modules of
\[ E'_{n-j+m} \otimes P_{r(m)}(v_c)\{t^{-p^{cm} k-\delta_c^m(\alpha^{(j-m+1)})} \lambda_c \}[-1] \]
for $1 \leq m \leq j$.

Finally, if $c=2$ the third family of small spectral sequences looks as follows. Recall that $E_j''=E(u_j,\lambda_2)$. For each $1 \leq j \leq n$ and each $k$ with $\nu_p(k-d_j(\alpha)+d_{j+1}(\alpha))=0$ we have a corresponding diagram, where the right hand side consists of submodules of
\[ E''_{n-j+m} \otimes P_{r(m+1)/p}(v_2)\{t^{d_{j-m}(\alpha)} \mu_2^{p^{2m} k+d_{j-m}(\alpha)+\delta_2^m(\alpha^{(j-m+1)})} \lambda_1 \} \]
for $0 \leq m \leq j$ and the left hand side consists of quotient modules of
\[ E''_{n-j+m} \otimes P_{r(m)/p}(v_2)\{t^{-p^{2m} k-\delta_2^m(\alpha^{(j-m+1)})} \lambda_1 \}[-1] \]
for $1 \leq m \leq j$.

The following theorem gives an algorithm for computing the homotopy orbit to $\TR$ spectral sequence. The expression for $d_\rho(x)$ looks unpleasant, but for $c=1$ or $2$ the formula, in the case when $d_\rho(x)$ is nontrivial, can be obtained simply from degree considerations.

\begin{thm} \label{thm:hotoTRSS}
The homotopy orbit to $\TR$ spectral sequence
\[ E^1_{s,t}(\alpha) = V_* T[-\alpha^{(n-s)}]_{hC_{p^s}} \Longrightarrow \TR^{n+1}_{\alpha+*}(A;V) \]
splits as a direct sum of the above spectral sequences, with no differentials between summands.

The differentials are determined by the following data. Let $e_j = u_j^{\epsilon_0} \lambda_1^{\epsilon_1} \lambda_2^{\epsilon 2}$ and suppose
\[ x = v_c^i t^{d_j(\alpha)} e_{n-j} \mu_c^k[-\alpha^{(j)}] \]
is a nontrivial class in the homotopy fixed point piece of $E^1_{n-j,*}$. Let
\[ y_h = v_c^{i^{(h)}} t^{-p^{hc}k-\delta_c^h(\alpha^{(j-h+1)})} e_{n-j+h}[-\alpha^{(j-h)}], \]
where
\[ i^{(h)} = i-r(h-1)k-\sum_{0 \leq k \leq h-2} \big[ d_{j-h+1}(\alpha)-d_j(\alpha) \big] p^{ck}.\]
If $x$ survives to $E^\rho_{n-j,*}$ and the classes $y_h \in V_* T[-\alpha^{(j-h)}]^{tC_{p^{n-j+h}}}$ are nonzero for $1 \leq h \leq \rho$ then $d_\rho(x)=\partial^h(y_\rho)$ considered as a class in $E^\rho_{n-j+\rho,*}$. If at least one of the classes $y_h$ for $1 \leq h \leq \rho$ is zero then $d_\rho(x)=0$.
\end{thm}

\begin{proof}
If we are not in the case $c=1$, $p=2$, then $\TR^k_{\beta+*}(A; V)$ is a module over $\TR^k_*(A; V)$, which contains an element $\mu_c^N$ for $N$ a multiple of $p^{c(k-1)}$. In the case $c=1$, $p=2$, $\TR^k_{\beta+*}(\bZ, V(0))$ is a module over $\TR^k_*(\bZ, S/4)$, which contains an element $\mu_1^N$ for $N$ a multiple of $2^k$. This follows by induction, using the results in \cite{Ro99b} and Tsalidis' theorem. In all cases we have a way of comparing with the stable range by multiplying by $\mu_c^N$ for an appropriate $N$.

The class $x$ in $V_* T[-\alpha^{(j)}]_{hC_{p^{n-j}}}$ maps to a class with the same name in $\TR^{n-j+1}_{\alpha^{(j)}+*}(A;V)$. By comparing with the stable range we find that $\hat{\Gamma}_{n-j+1}(x)=y_1$. By construction of the spectral sequence this implies that $d_1(x)=\partial^h(y_1)$.

If $d_1(x)=0$, then $x$ lifts to a class $x_1$ in $\TR^{n-j+2}_{\alpha^{(j-1)}+*}(A;V)$. Let $z_1=\Gamma_{n-j+1}(x_1)$ in $V_* T[-\alpha^{(j-1)}]^{hC_{p^{n-j+1}}}$. While $x_1$, and hence $z_1$, may not be unique, we have a canonical choice for a representative for $z_1$ in the homotopy fixed point spectral sequence given by taking a representative for $\hat{\Gamma}_{n-j+1}(x)$ in the Tate spectral sequence and restricting to the second quadrant. We then have two cases.

Case 1: The class $z_1$ multiplies nontrivially by $\mu_c^N$ to the stable range. Because $V_* T[-\alpha^{(j-1)}]^{hC_{p^{n-j+1}}}$ is isomorphic to $V_* T[-\alpha^{(j-2)}]^{tC_{p^{n-j+2}}}$ in the stable range, this happens exactly when $y_2=\hat{\Gamma}_{n-j+2}(x_1) \neq 0$. Again it follows by construction of the spectral sequence that $d_2(x)=\partial^h(y_2)$. The formula for $d_\rho(x)$ assuming $y_1,\ldots,y_\rho$ are all nonzero follows by induction.

Case 2: The class $z_1$ multiplies trivially by $\mu_c^N$ to the stable range. In this case we find that $\hat{\Gamma}_{n-j+2}(x_2)=0$, so $d_2(x)=0$.  By induction, $x$ lifts to a class $x_h$ in $\TR^{n-j+h+1}_{\alpha^{(j-h)}+*}(A;V)$ which multiplies trivially to the stable range for all $h$. Hence $d_\rho(x)=0$ for all $\rho$. The same argument applies as soon as some $y_h$ is zero.

\end{proof}

\section{The $RO(S^1)$-graded TR--groups of $\bF_p$} \label{s:TRFp}
While Theorem \ref{thm:hotoTRSS} above tells us all the differentials in the spectral sequence converging to $\TR^{n+1}_{\alpha+*}(\bF_p)$, we need some additional information to resolve the extension problems. As shown in \cite{Ge07}, the extension problem is in fact quite delicate.

We observe that if we know the order of $\TR^{n+1}_{\alpha+*}(\bF_p;\bZ/p^l)$ for each $l \geq 1$, we can reconstruct $\TR^{n+1}_{\alpha+*}(\bF_p)$. Let $T=T(\bF_p)$. We find that
\[ \pi_* (T[-\alpha];\bZ/p^l) \cong t^{d_0(\alpha)} E(\beta_l) \otimes P(\mu_0),\]
where $\beta_l$ is in degree $1$, and the Tate spectral sequence behaves as follows:

\begin{lemma}
Consider the spectral sequence converging to $\pi_*(T[-\alpha]^{tC_{p^n}};\bZ/p^l)$. If $n<l$ there is a differential $d_{2n+1}(u_n) = t v_0^n$ and if $n \geq l$ there is a differential $d_{2l}(\beta_l)=v_0^l$.
\end{lemma}

\begin{proof}
This is clear because the mod $p^l$ Bockstein $\beta_l$ will always kill the representative for $p^l$ if possible.
\end{proof}

We can then record $\pi_*(T[-\alpha]_{hC_{p^n}};\bZ/p^l)$. As before, we split it into the Tate piece and the homotopy fixed point piece. If $n<l$ we find that the Tate piece is
\begin{eqnarray*}
& \displaystyle \bigoplus_{k \geq n-d_0(\alpha)} & E(\beta_l) \otimes P_n(v_0)\{t^{-k}[-\alpha]\}[-1] \\
& \displaystyle \bigoplus_{1 \leq k+d_0(\alpha) \leq n-1} & E(\beta_l) \otimes  P_{k+d_0(\alpha)}(v_0)\{t^{-k}[-\alpha]\}[-1].
\end{eqnarray*}
Similarly, the homotopy fixed point piece is as follows:
\begin{eqnarray*}
& \displaystyle \bigoplus_{k \geq n} & E(\beta_l) \otimes P_{n+1}(v_0)\{t^{d_0(\alpha)} \mu_0^k[-\alpha]\} \\
& \displaystyle \bigoplus_{1 \leq k \leq n} & E(\beta_l) \otimes P_k(v_0)\{v_0^{n+1-k} t^{d_0(\alpha)} \mu_0^{k-1}[-\alpha]\}
\end{eqnarray*}
If $n \geq l$ we find that the Tate piece is
\begin{eqnarray*}
& \displaystyle \bigoplus_{k \geq l-d_0(\alpha)} & E(u_n) \otimes P_l(v_0)\{t^{-k}[-\alpha]\}[-1] \\
& \displaystyle \bigoplus_{1 \leq k+d_0(\alpha) \leq l-1} & E(u_n) \otimes  P_{k+d_0(\alpha)}(v_0)\{t^{-k}[-\alpha]\}[-1].
\end{eqnarray*}
Similarly, the homotopy fixed point piece is as follows:
\begin{eqnarray*}
& \displaystyle \bigoplus_{k \geq l} & E(u_n) \otimes P_l(v_0)\{t^{d_0(\alpha)} \mu_0^k[-\alpha]\} \\
& \displaystyle \bigoplus_{1 \leq k \leq l-1} & E(u_n) \otimes P_k(v_0)\{v_0^{l-k} t^{d_0(\alpha)} \mu_0^k[-\alpha]\}
\end{eqnarray*}

\begin{thm}
Consider the spectral sequence
\[ E^1(\alpha)= \bigoplus_{0 \leq s \leq n} \pi_*(T[-\alpha^{(n-s)}]_{hC_{p^s}};\bZ/p^l) \Longrightarrow \TR^{n+1}_{\alpha+*}(\bF_p;\bZ/p^l). \]
The differentials are determined by the following data. Suppose
\[ x = p^i t^{d_j(\alpha)} u_{n-j}^\epsilon \mu_0^k[-\alpha^{(j)}] \]
is a nontrivial class in the homotopy fixed point piece of $E^1_{n-j,*}$. Then $d_\rho(x)$ is given as in Theorem \ref{thm:hotoTRSS}.

Now suppose
\[ \bar{x} = p^i t^{d_j(\alpha)} \beta_l \mu_0^k[-\alpha^{(j)}] \]
is a nontrivial class in the homotopy fixed point piece of $E^1_{n-j,*}$, and
\[ \bar{y}_h = \begin{cases} p^{i^{(h)}} t^{-k-\delta_0^h(\alpha^{(j-h+1)})} \beta_l[-\alpha^{(j-h)}] & \text{if $n-j+h-l<0,$} \\
p^{i^{(h)}-(n-j+h-l)} t^{-k-1-\delta_0^h(\alpha^{(j-k+1)})} u_{n-j+h}[-\alpha^{(j-h)}] & \text{if $n-j+h-l \geq 0.$} \end{cases} \]
If $\bar{x}$ survives to $E^\rho_{n-j,*}$ and the classes $\bar{y}_h \in \pi_*(T[-\alpha^{(j-h)}]^{tC_{p^{n-j+h}}};\bZ/p^l)$ are nonzero for $1 \leq h \leq \rho$ then $d_\rho(\bar{x})=\partial^h(\bar{y}_\rho)$ considered as a class in $E^\rho_{n-j+\rho,*}$. If at least one of the classes $\bar{y}_h$ is zero then $d_\rho(\bar{x})=0$.
\end{thm}

\begin{proof}
The proof is similar to the proof of Theorem \ref{thm:hotoTRSS}. The extra factor of $p^{-(n-j+h-l)}$ comes from having $n-j+h-l$ homotopy orbit spectral sequences with a differential on $\beta_l$ rather than a differential on some $u_{j+h}$. For each one, the possible differential, and possible successive lift of $\bar{x}$, behaves as if we had started with a multiple of $u_{n-j} \mu_0^{k+1}[-\alpha^{(j)}]$ rather than a multiple of $\beta_l \mu_0^k[-\alpha^{(j)}]$.
\end{proof}

\section{The TR groups in degree $q-\lambda$} \label{s:TRq-lambda}
It is the TR--groups indexed by representations of the form $\alpha = q - \lambda$ that are most applicable to computations of algebraic $K$-theory. See, for example, Hesselholt and Madsen's computation of $K_q(\mathbb{F}_p[x]/(x^m), (x))$ in \cite{He05} and results of the authors and Hesselholt on $K_q(\mathbb{Z}[x]/(x^m), (x))$ in \cite{AnGeHe}.

\begin{proposition}\label{prop:differentials}
Consider the spectral sequence
\[ E^1_{s,t}(-\lambda) = \bigoplus_{0 \leq s \leq n} V_*T[-\lambda^{(n-s)}]_{hC_{p^s}} \Longrightarrow \TR^{n+1}_{*-\lambda}(A;V) \]
for an actual representation $\lambda$. Then every nonzero class in the Tate piece is killed by a differential.
\end{proposition}

\begin{proof}
We prove this by induction, but with a slightly extended induction hypothesis. We consider a representation $\lambda$ which is \emph{almost} an actual representation, by which we mean that $d_i(\lambda) \geq d_{i+1}(\lambda)$ for $i \geq 1$ and $d_0(\lambda) \geq d_1(\lambda)-1$.

Consider the first family of spectral sequences described in \S \ref{s:hotoTRSSVc-1}. It is enough to show that
\[ z = t^{-p^{cn} k+\delta_c^n(\lambda')}[-1] \]
in the Tate piece of $E^1_{n,*}(-\lambda)$ is hit by a differential. For $z$ to be nonzero we must have
\[ p^{cn} k-d_0(\lambda)-\delta_c^n(\lambda') \geq 0.\]
Consider
\[ x = t^{-d_1(\lambda)} \mu_c^{p^{c(n-1)}k - d_1(\lambda) - \delta_c^{n-1}(\lambda'')} \]
in the homotopy orbit piece of $E^1_{n-1,*}$. If
\[ p^{c(n-1)}k-d_1(\lambda)-\delta_c^{n-1}(\lambda'') > r(n-2) \]
then $x$ is nonzero and $d_1(x)=z$.

Now suppose
\[ p^{c(n-1)}k-d_1(\lambda)-\delta_c^{n-1}(\lambda'') \leq r(n-2).\]
Consider the class
\[ y = v_c^{p^{c(n-1)}k-d_1(\lambda)-\delta_c^{n-1}(\lambda'')} t^{-p^{c(n-1)}k+\delta_c^{n-1}(\lambda'')} \]
in the Tate spectral sequence converging to $V_* T[-\lambda']^{tC_{p^{n-1}}}$. Then $y$ is in filtration $2d_1(\lambda)$, which means that $y$ is not in the first quadrant of the spectral sequence and hence $\partial^h(y)=0$. Note that
\[ 0 \leq p^{c(n-1)}k-d_1(\lambda)-\delta_c^{n-1}(\lambda'') \leq r(n-2),\]
so $y$ is nonzero in $V_* T[-\lambda']^{tC_{p^{n-1}}}$.

By assumption, $d_1(\lambda) \geq d_2(\lambda)$. Then we can consider a representation $\mu$ with $\mu''=\lambda''$ and $d_1(\mu)=d_1(\lambda)-1$. Then $\partial^h(y) \neq 0$ in $E^1_{n-1,*}(-\mu)$. By induction $\partial^h(y)=d_\rho(w)$ for some $w$ in $E^\rho_{*,*}(-\mu)$. But then Theorem \ref{thm:hotoTRSS} implies that $d_{\rho+1}(w)=z$ in $E^{\rho+1}_{*,*}(-\lambda)$, proving the result.

The remaining two families of differentials can be treated in a similar way.
\end{proof}

\noindent
We can now redo the calculation in \cite{HeMa97}:

\begin{corollary}
It follows that
\[ |\TR^n_{q-\lambda}(\bF_p)| = \begin{cases} p^n & \text{for $q=2m$ and $d_0(\lambda) \leq m$,} \\
p^{n-s} & \text{for $q=2m$ and $d_s(\lambda) \leq m \leq d_{s-1}(\lambda)$,} \\
0 & \text{for $q$ odd.} \end{cases} \]
\begin{proof}
In the case of $\mathbb{F}_p$, if we consider the spectral sequence
\[ E^1_{s,*}(-\lambda) = \pi_*T[-\lambda^{(n-1-s)}]_{hC_{p^s}} \Longrightarrow \TR^{n}_{*-\lambda}(\bF_p),\]
the only elements in odd total degree are those in the Tate piece. By Proposition \ref{prop:differentials}, all those elements are killed, hence $|\TR^n_{q- \lambda}(\mathbb{F}_p)| = 0$ for $q$ odd. In even degrees, since the differentials are surjective
\[ |\TR^n_{2m-\lambda}(\bF_p)| = \frac{\prod_s |E^1_{s, 2m}|}{\prod_s |E^1_{s,2m-1}|} =
\begin{cases}
p^n & \text{for $d_0(\lambda) \leq m,$} \\
p^{n-s} & \text{for $d_s(\lambda) \leq m \leq d_{s-1}(\lambda).$}
\end{cases} \]

\end{proof}
\end{corollary}
From the spectral sequence for $\TR^n_{q-\lambda}(\bF_p; \bZ/p)$ in \S \ref{s:TRFp} we conclude that $\TR^n_{q-\lambda}(\bF_p)$ has just one summand. So we get the following result:

\begin{theorem}
Let $\lambda$ be a finite complex $S^1$-representation. Then
\[ \TR^n_{q- \lambda}(\mathbb{F}_p) \cong
\begin{cases}
\bZ/p^n & \text{for $q=em$ and $d_0(\lambda) \leq m$,} \\
\bZ/p^{n-s} & \text{for $q=2m$ and $d_s(\lambda) \leq m \leq d_{s-1}(\lambda)$,} \\
0 & \text{for $q$ odd.}
\end{cases} \]
\end{theorem}

This agrees with the result of Hesselholt and Madsen \cite{HeMa97}. In the case of $A = \mathbb{Z}$ we can then prove Theorem \ref{thm:mainL}:

\begin{proof}[Proof of Theorem \ref{thm:mainL}]
As described in \S \ref{s:hotoTRSSVc-1} the $E^1$-term of the homotopy orbit to TR spectral sequence is composed of two families of small spectral sequences. In sufficiently high degrees we are left with the lower right-hand summands in the diagrams of \S \ref{s:hotoTRSSVc-1}. We first give the argument in high degrees and then describe the modifications needed in low degrees.


In the $E_\infty$-term we are left with
\[ E(\lambda_1) \otimes P_{r(n-1) + 1}(v_1)\{t^{-d_0(\lambda)} \mu_1^{p^{n-1}k - d_0(\lambda)-\delta^{n-1}_1(\lambda')}\} \]
from the first family of spectral sequences, and
\[ E(u_j) \otimes P_{r(j+1)}(v_1)\{t^{-d_0(\lambda)} \mu_1^{p^jk - d_0(\lambda)-\delta^j_1(\lambda')} \lambda_1 \} \]
for $0 \leq j \leq n-2$ and $k$ such that $v_p(k + d_j(\lambda)+\delta^{n-1-j}_1(\lambda^{(j+1)})) = 0$ from the second family.
Assume $q = 2m$ is even. The length of $\TR^n_{q - \lambda}(\mathbb{Z}; \mathbb{Z}/p)$ is the number of different ways can $2m$ be written as
\[ 2m = 2d_0(\lambda)+ 2 p^nk - 2p(d_0(\lambda)+\delta^{n-1}_1(\lambda')) + a(2p-2) \]
for $0\leq a \leq r(n-1)$ or
\[ 2m = 2d_0(\lambda)+ 2p^{j+1}k - 2p(d_0(\lambda)+\delta^j_1(\lambda')) + (a+1)(2p-2) \]
for $0 \leq j \leq n-2$,  $0 \leq a < r(j+1)$, and $v_p(k + d_j(\lambda)+\delta^{n-1-j}_1(\lambda^{(j+1)})) = 0$. Noting that
\[ \delta_1^j(\lambda') = \delta_1^{n-1}(\lambda') - p^j d_j(\lambda)-p^j \delta_1^{n-1-j}(\lambda^{(j+1)}) \]
we can rewrite these two equations as
\[ 2m - 2d_0(\lambda) + 2p(d_0(\lambda)+\delta^{n-1}_1(\lambda')) = 2p^nk + a(2p-2) \]
or
\begin{multline*} 2m - 2d_0(\lambda)+ 2p(d_0(\lambda)+\delta^{n-1}_1(\lambda')) \\
= 2p^{j+1}(k+ d_j(\lambda)+\delta^{n-j-1}_1(\lambda^{(j+1)})) + (a+1)(2p-2)
\end{multline*}
with the same conditions on $a, j$, and $k$ as above. It follows that the length of $\TR^n_{2m - \lambda}(\mathbb{Z};  \mathbb{Z}/p)$ is the number of ways to write $b=m-d_0(\lambda)+p(d_0(\lambda)+\delta_1^{n-1}(\lambda'))$ as
\[ b = p^nk + a(p-1) \]
where $0 \leq a \leq r(n-1)$ or
\[ b = p^{j+1}k + a(p-1) \]
where $0 \leq j \leq n-2$,  $1 \leq a \leq r(j+1)$, and $v_p(k) = 0$. Now, if $b=p^n k +a(p-1)$ with $1 \leq a \leq r(n-1)$ we can rewrite this as $b=p^{n-1} (pk)+a(p-1)$, and if $b=p^{j+1} k+a(p-1)$ with $1 \leq a \leq r(j)$ we can rewrite it as $b=p^j (pk)+a(p-1)$. Hence we have one class when $c=0$ modulo $p^n$ and one class for each way to write
\[ b = p^{j+1} k+a(p-1) \]
with $0 \leq j \leq n-2$ and $r(j)<a\leq r(j+1)$, with no condition on $\nu_p(k)$. There is exactly one such pair $(k,a)$ for each $j$, so we get $n-1$ classes, plus an additional class from the first family when $m=\delta_1^n(\lambda)$ modulo $p^n$ corresponding to $a=0$. The case $q=2m+1$ odd is similar.

If $q \geq 2d_0(\lambda)$, but $q$ is not sufficiently high that the spectral sequences degenerate with only the lower right hand summands in the $E^\infty$ term, the result follows by comparing with $\pi_*(T[-\mu]^{tC_{p^n}};\bZ/p)$ for some $\mu$ with $\mu'=\lambda$. Using that the mod $p$ homotopy groups of the Tate spectrum are $2p^n$-periodic and Theorem \ref{t:Tsalidis}, the result follows.


Part $2$ and $3$ follow by using that if $q<2d_0(\lambda)$ we have an isomorphism
\[ R : \TR^n_{q-\lambda}(\bZ;\bZ/p) \overset{\cong}{\to} \TR^{n-1}_{q-\lambda'}(\bZ;\bZ/p).\]

%
%

\end{proof}

\bibliographystyle{plain}
\bibliography{b}

\end{document}